\documentclass{article}
\usepackage[utf8]{inputenc}
\usepackage{amsmath}
\usepackage[T1]{fontenc}    
\usepackage{hyperref}       
\usepackage{url}            
\usepackage{booktabs}       
\usepackage{amsfonts}       
\usepackage{nicefrac}       
\usepackage{microtype}      
\usepackage{lipsum}
\usepackage{graphicx}
\usepackage[english]{babel}
\usepackage{amsthm}
\usepackage{tabularx}
\usepackage{array}

\graphicspath{ {./images/} }

\usepackage{mathtools}

\usepackage{amssymb}

\newtheorem{theorem}{Theorem}[section]

\newtheorem{lemma}[theorem]{Lemma}
\newtheorem{definition}[theorem]{Definition}

\usepackage{listings}
\usepackage{xcolor}
\title{ Bounds for the Number of Terms of Harmonic Sums}
\author{Keneth Adrian P. Dagal \\ 
  \texttt{kendee2012@gmail.com} \\}

\begin{document}

\maketitle

\begin{abstract}
This paper provides bounds for the number of terms, denoted by $f$, of a harmonic sum with the condition that it starts from any arbitrary unit fraction $\frac{1}{m}$, $m > 1$,  until another unit fraction $\frac{1}{m+f-1}$ such that the sum is the highest sum less than a particular positive integer $q$. Also, we consider the number of terms of  Egyptian fractions whose terms are consecutive multiples of $r$, $r \geq 1$,  under the same above condition. We end the paper with a formula for the case: $q=1$ and $r=1$.
\end{abstract}

\section{Introduction}

We start with the function defined by Dagal \cite{DagalOp}.

\begin{definition}
The function $S$ is said to be an Egyptian fraction if $$S(X) = \sum_{x \in X }{x}.$$ for all $X \in \mathcal{P}(U_f)_{\mid X\mid \geq 2}$.
\end{definition}

To illustrate the definition, suppose we have $N= \{2,3,6\}$, then $X = \{\frac{1}{2}, \frac{1}{3}, \frac{1}{6}\}$. And thus, $$S(X) = \frac{1}{2}+ \frac{1}{3}+ \frac{1}{6}=1.$$

We notice that $N$ can be of any finite collection of positive integers greater than 1 and with cardinality greater than 1. We \textit{restrict} this collection with the property that each element of $N$ are consecutive multiples of $r$, $r\geq 1$. If $r=1$, then the elements of $N$ are said to be consecutive numbers, and  the function $S$ now becomes a \textit {Harmonic Sum}, denoted by $Q_m^n$.

We know that the $n$th Harmonic number is defined as:

$$H_n = 1+\frac{1}{2}+ \frac{1}{3}+ \cdots + \frac{1}{n} = \sum_{i=1}^{n} \frac{1}{i}.$$

We focus on the expression $H_n-1$ since the terms of the said harmonic sum are unit fractions. We define formally the function $Q_m^n$.

\begin{definition}
The function $Q_m^n$ is said to be a Harmonic Sum if $$Q_m^n = \sum_{i=m}^{m+f-1}{\frac{1}{i}}.$$ where $m \geq 2$, and $n = m+f-1$, $f \geq 2$. $f$ is the number of terms of the sum.
\end{definition}

The image of this function is clearly in $\mathbb{Q^+}$ and we know that $\mathbb{Z^+} \subset \mathbb{Q^+}$. Evidently, any $H_n$ or any difference of harmonic sums, denoted by $Q_m^n$, are in $\mathbb{Q^+}$. Clearly, $$Q_m^n = H_n- H_{m-1}.$$

A question of Tavares \cite{Tavares} about harmonic sums is answered by Bill Duduque and an anonymous person. The source of the question mentioned in the page is  Cf. Exercise 6.21 (page 311) of Graham, Knuth, and Patashnik and the answered question is given below:

\begin{theorem}
Let $H_{m-1}$ and $H_{n}$ be harmonic sums for positive integers $m$ and $n$, $1 < m < n$. Then $Q_{m}^{n}$, $H_{m-1}$, and $ H_n$ are not integers.
\end{theorem}

The theorem above guarantees that $Q_{m}^{n}\notin \mathbb{Z^+}$ but $ \in \mathbb{Q^+}$. Also, in the same page, it was stated that $H_n$ cannot be an integer for $n > 1$. Thus, we can find an integer $q$ with $$Q_{m}^{n} < q < Q_{m}^{n+1}.$$ Instead of finding $q$, we fix $q$ and find the number of terms starting from a unit fraction, denoted by $(m)$, where $m$ is the denominator of the unit fraction,  up to, but not equal to, a desired integer $q$. For the remainder of the paper, we set $m > 1$.

We now set $q= 1$, and have the condition: $Q_{m}^{n} < 1 < Q_{m}^{n+1}.$ Our starting unit fraction is $(m)$ until $(m+f-1)$ such that the condition holds. With this, we have the inequality below
$$Q_{m}^{m+f-1} < 1 < Q_{m}^{m+f}.$$

The question now is: what is the value of $f$ such that the said condition holds? In other words, how many terms are there in the harmonic sum $Q_{m}^{m+f-1}$ before it exceeds 1? 

\section{The Bounds for Number of Terms of the Harmonic Sum}

To answer the question in the previous section, we start with using the definite integral  $$\int_m^{m+1}\frac{1}{x} dx = \ln\left(\frac{m+1}{m}\right),$$
 the inequality $$\frac{1}{m+1} < \int_m^{m+1}\frac{1}{x} dx \,< \frac{1}{m},$$ and by revisiting Lemma 2.3 of Dagal \cite{DagalAlo} and then revise it here for our use. So, we state the lemma below:
 
\begin{lemma}
Let $Q_m^n$ and $Q_{m+1}^{n+1}$ be the harmonic sums. Then the inequality $$ Q_{m+1}^{n+1} <  \ln \left(\frac{n+1}{m}\right) < Q_{m}^{n}$$ holds.
\end{lemma}

\begin{proof}
The inequality $$\frac{1}{m+1} < \int_m^{m+1}\frac{1}{x} dx \,< \frac{1}{m},$$ holds for any positive integer $m$.
If $m=2$, then $$\frac{1}{3} < \int_2^{3}\frac{1}{x} dx \, = \ln\left(\frac{3}{2}\right)< \frac{1}{2}.$$ By the definite integral, we can rewrite the inequality as

$$\frac{1}{m+1} < \ln\left(\frac{m+1}{m}\right) < \frac{1}{m},$$
We increment from $m$ until $n$ and add them, then we have

$$\sum_{i=m+1}^{n+1}\frac{1}{i} < \ln\left(\frac{m+1}{m}\cdot \frac{m+2}{m+1} \cdots \frac{m+f-1}{m+f-2} \cdot \frac{m+f}{m+f-1}\right) \, < \sum_{i=m}^{n}\frac{1}{i}, $$
where $n = m+f-1.$ Thus, simplifying and by definition of harmonic sum, we get
$$ Q_{m+1}^{n+1} <  \ln \left(\frac{n+1}{m}\right) < Q_{m}^{n}.$$
\end{proof}

In the next theorem, we provide the bounds for number of terms, denoted by $f$.

\begin{theorem}
If $Q_{m}^{n} < 1 < Q_{m}^{n+1}$ where $n= m+f-1$ , then $$ \lceil m (e-1)-e\rceil \leq f(m) \leq \lfloor m(e-1) \rfloor.$$
\end{theorem}

\begin{proof}
By Lemma 2.1 and the given, we have 
$$ \ln \left(\frac{n+1}{m}\right) < Q_{m}^{n} < 1$$

Thus, by substitution and simplification, we have

$$ \ln \left(\frac{m+f}{m}\right) < 1$$
$$ \frac{m+f}{m} < e$$
$$ f< m(e-1)$$
Since $f$ is an integer, the closest integer above $f$ is the floor of the expression $m(e-1)$. Thus, the upper bound of $f$ is given below:
$$ f(m) \leq \lfloor m(e-1) \rfloor.$$

For the lower bound, we use the again lemma 2.1 and the given, which we have,
 $$1 < Q_{m}^{n+1} <  \ln \left(\frac{n+1}{m-1}\right)$$
 
 In similar fashion, we have
 
 $$1 <  \ln \left(\frac{m+f}{m-1}\right)$$
 $$e(m-1) <  m+f$$
 $$m(e-1)-e < f$$

Since $f$ is an integer, the closest integer below $f$ is the ceiling of the expression $m(e-1)-e$. Thus, the lower bound of $f$ is given below

$$ \lceil m (e-1)-e\rceil \leq f(m).$$

Thus the bounds for $f(m)$ is 

$$ \lceil m (e-1)-e\rceil \leq f(m) \leq \lfloor m(e-1) \rfloor.$$

\end{proof}

We now proceed to $q$ as a positive integer, in general, instead only of $q=1$ (which was given in the previous theorem). We have the bounds below.

\begin{theorem}
If $Q_{m}^{n} < q < Q_{m}^{n+1}$ , where $n= m+f-1$, then 

$$ \lceil m (e^q-1)-e^q\rceil \leq f(m, q) \leq \lfloor m(e^q-1) \rfloor.$$
\end{theorem}

\begin{proof}
The proof is straightforward by simply mimicking the proof of Theorem 2.2 and replacing $1$ by $q$.
\end{proof}

We add another parameter $r$ where $r$ denotes the $r$th  consecutive multiples of the unit fraction of the harmonic sum $Q_m^n$. Take note that we now have a new harmonic sum with the starting unit fraction $ (m\cdot r)$ instead of $m$.

With this, we state the theorem below:

\begin{theorem}
If $\frac{1}{r}Q_{m}^{n} < q < \frac{1}{r}Q_{m}^{n+1}$ , where $n= m+f-1$, then 
$$ \lceil m (e^{qr}-1)-e^{qr}\rceil \leq f(m, q, r) \leq \lfloor m(e^{qr}-1) \rfloor.$$

\end{theorem}

\begin{proof}
In line with the previous proofs along with introducing $r$, we have $$ q < \frac{1}{r}Q_{m}^{n+1}\Rightarrow qr < Q_{m}^{n+1} \text{ and } \frac{1}{r}Q_{m}^{n} < q  \Rightarrow Q_{m}^{n} < qr $$
From here, we can proceed to show the bounds

$$ \lceil m (e^{qr}-1)-e^{qr}\rceil \leq f(m, q, r) \leq \lfloor m(e^{qr}-1) \rfloor$$

same as the proofs above.
\end{proof}

\section{Illustration of the Theorems}

In theorem 2.4, we have the bounds:
$$ \lceil m (e^{qr}-1)-e^{qr}\rceil \leq f(m, q, r) \leq \lfloor m(e^{qr}-1) \rfloor$$

for the number of terms $f$, with independent parameters $m$, $q$, and $r$.

By letting $A = m (e^{qr}-1)-e^{qr}$, we can rewrite the inequality
$$ \lceil A\rceil \leq f(m, q, r) \leq \lfloor A+ e^{qr}\rfloor$$

Notice that the range of the candidates for $f$ is within $e^{qr}$. If we have $q=1$ and $r=1$, then the range becomes $e$. With this information, we can have a theorem below:

\begin{theorem}
Let $E = m (e-1)$. Then $$f(m) = \left\lfloor E-\frac{e}{2}\right\rfloor \text{ or } \left\lfloor E-\frac{e}{2}\right\rfloor+1 $$

\end{theorem}

\begin{proof}
In theorem 2.2, we have the inequality 

$$ \lceil m (e-1)-e\rceil \leq f(m) \leq \lfloor m(e-1) \rfloor.$$

and rewrite the previous inequality as:

$$ \lceil E-e\rceil \leq f(m) \leq \lfloor E \rfloor.$$

Finding the middle value of the range we have $E- \frac{e}{2}$. Thus,

$$ \lceil E-e\rceil \leq f(m) < E - \frac{e}{2} < f(m) \leq \lfloor E \rfloor.$$

Since the expression $E - \frac{e}{2}$ can not be an integer due to the irrationality of $e$, then

$$\left\lceil E - \frac{e}{2}\right\rceil- \left\lfloor E- \frac{e}{2}\right\rfloor  = 1$$

With this, we can deduce the following:

$$ 0 \leq \lfloor E \rfloor-\left\lceil E - \frac{e}{2}\right\rceil < \frac{e-1}{2} < 1$$
$$0 \leq \left\lfloor E- \frac{e}{2}\right\rfloor -\lceil E- e\rceil < \frac{e-1}{2} < 1$$

Clearly, if the inequality below is true, then we are done.

$$ \lceil E-e\rceil = f(m)=\left\lfloor E- \frac{e}{2}\right\rfloor  \text{ or }\left\lceil E - \frac{e}{2}\right\rceil = f(m) = \lfloor E \rfloor.$$

Or else, we have $$f(m)=\left\lfloor E- \frac{e}{2}\right\rfloor  \text{ or }\left\lceil E -\frac{e}{2}\right\rceil = f(m)$$ is true in either of the following case: 

$$ \lceil E-e\rceil  < \left\lfloor E- \frac{e}{2}\right\rfloor  \text{ or }\left\lceil E - \frac{e}{2}\right\rceil < \lfloor E \rfloor.$$

In summary, 

$$f(m) = \left\lfloor E - \frac{e}{2}\right\rfloor \text{ or } \left\lfloor E - \frac{e}{2}\right\rfloor+1 $$

\end{proof}

We can now generalize the previous theorem for positive integers $q$ and $r$, not necessarily 1. As these two number increases,  the range of the bounds also increases but the real value might be closer to the midpoint. We now state the theorem. 

\begin{theorem}
Let $F = m (e^{qr}-1)$. Then $$  \lceil F-e^{qr}\rceil \leq f(m,q,r) \leq \left\lfloor F- \frac{e^{qr}}{2}\right\rfloor \text{ or } 
\left\lfloor F- \frac{e^{qr}}{2}\right\rfloor+1 \leq  f(m,q,r) \leq  \lfloor F \rfloor.$$

\end{theorem}

\begin{proof}
The proof of this theorem is of similar fashion as the previous one. Since $e^{qr}$ is the range, Then the candidates for the value of $f$ is within this range.
\end{proof}

Having bounds provides a set of numbers that can be a candidate for the function. As such, we illustrate Theorem 3.1 by using a C++ program below.

\begin{lstlisting}
#include <iostream>
#include <cmath>
using namespace std;

int main()
{
int i=0;
double s,m,q,r;
cout << "begin at:"; cin>> m; 
cout << "sum up to:"; cin >> q;
cout << "multiple of:"; cin >> r;

 cout << "Lower Bound:" << ceil((exp(q*r)-1)*m- exp(q*r))  << "\n";
 cout << "Upperbound :" << floor((exp(q*r)-1)*m) << "\n";
 cout << "midlow:" << floor((exp(q*r)-1)*m- (exp(q*r)/2))<< "\n";
 cout << "midhigh:" << ceil((exp(q*r)-1)*m- (exp(q*r)/2))<< "\n";
 cout << "Real Value:";
  do{s= s + 1/(r*m);i++;m++;} while( s < q); cout<< i-1;
    
 return 0;
    
}
\end{lstlisting}

To summarize, we have the table below to show some particular examples of the theorems stated in this paper with the use of the program shown above.

\begin{center}
    \begin{tabular}{|c|c|c|c|c|c|c|c|c|}
    \hline
       $EN$ & $m$ & $q$ & $r$ & $LB$ & $ML$ & $RV$ & $MH$ & $UB$\\ \hline
       1& 5 & 1 & 1 & 6  & 7 & 7 & 8 & 8\\ \hline
       2& 11 & 1 & 1 & 17  & 17& 18 & 18 & 18\\ \hline
       3& 1000 & 1 & 1 & 1716  & 1716& 1717 & 1717 & 1718\\ \hline
       4& 23 & 2 & 3 & 8853  & 9054 & 9055& 9055 & 9255\\ \hline
       5&  5 & 5 & 1 & 589  & 662& 664 & 663 & 737\\ \hline
       6& 100,000 & 1 & 1 & 171826  & 171826 & 171827 & 171827 & 171828\\ \hline
       7& 2 & 1 & 10 & 22025  & 33037& 33615 & 33038 & 44050\\ \hline
       8& 3 & 3 & 3 & 16204  & 20254& 20387 & 20255 & 24306\\ \hline
       9& 3 & 10 & 1 & 44050  & 55063& 55422 & 55064 & 66076\\ \hline
       10& 105 & 1 & 1 & 178  & 179& 179 & 180 & 180\\ \hline
    \end{tabular}
\end{center}

 The Legend for the table above:
 
 EN= Example Number. LB= Lower bound. ML= Middle Low. RV= Real Value. MH= Middle High. UB= Upperbound.

Let us expound on some of the working examples above.

Take Example no 1. $m=5$, $q=1$, $r=1$. $f(5,1,1) = 7$. Thus,

$$\frac{1}{5}+ \frac{1}{6}+ \cdots +\frac{1}{11} < 1 < \frac{1}{5}+ \frac{1}{6}+ \cdots +\frac{1}{12} $$

$$0.93654... < 1 < 1.01988... $$

Take Example no 6. $m=100000$, $q=1$, $r=1$. $f(100000,1,1) = 171827$. Thus,

$$\frac{1}{100000}+ \frac{1}{100001}+ \cdots +\frac{1}{271826} < 1 < \frac{1}{100000}+ \frac{1}{100001}+ \cdots +\frac{1}{271827} $$
Using wolfram-alpha at https://www.wolframalpha.com/, we have

$$0.9999... < 1 < 1.0000... $$

Take Example no 5. $m=5 $, $q=5$, $r=1$. $f(5,5,1) = 664$. Thus,

$$\frac{1}{5}+ \frac{1}{6}+ \cdots +\frac{1}{668} < 5 < \frac{1}{5}+ \frac{1}{6}+ \cdots+\frac{1}{669} $$

$$ 4.99891... < 5 < 5.00041... $$

Take Example no 7. $m=2 $, $q=1$, $r=10$. $f(2,1,10) = 33615$. Thus,

$$\frac{1}{20}+ \frac{1}{30}+ \frac{1}{40}+\cdots +\frac{1}{336160} < 1 < \frac{1}{20}+ \frac{1}{30}+ \frac{1}{40}+\cdots +\frac{1}{336170} $$

$$ 0.99999879... < 1 < 1.0000017...$$

And last, we can take Example no 8. $m=3 $, $q=3$, $r=3$. $f(3,3,3) = 20387$. Thus,

$$\frac{1}{9}+ \frac{1}{12}+ \frac{1}{15}+\cdots +\frac{1}{61167} < 3 < \frac{1}{9}+ \frac{1}{12}+ \frac{1}{15}+\cdots +\frac{1}{61170} $$

$$ 2.999997... < 3 < 3.00001335...$$
\section{Acknowledgement}

The author would like to thank Mohammad Abdul Aziz Qureshi for valuable time in verifying one formula introduced in this paper. Also, the author would like to thank Jose Arnaldo Dris for essential conversations. Ultimately, the author would like to thank the Almighty God for everything therein.


\begin {thebibliography}{1}


\bibitem{DagalAlo}
Dagal, K. (2015). A lower bound for $\tau(n)$ of any k-perfect numbers. Retrieved 2 April 2020, from http://www.m-hikari.com/pms/pms-2015/pms-1-4-2015/4923.html
\bibitem{DagalOp}
Dagal, K. (2020). A New Operator For Egyptian Fraction. Retrieved from https://arxiv.org/abs/2003.13229

\bibitem{Tavares}
Tavares, A. (2010). How do I prove this sum is not an integer. Retrieved 2 April 2020, from https://math.stackexchange.com/questions/5219/how-do-i-prove-this-sum-is-not-an-integer

\end{thebibliography}
\end{document}